\documentclass[12pt]{article}
\usepackage{amscd}
\usepackage{bm}
\advance \textwidth -60truemm\relax

\advance\oddsidemargin -1truein\relax

\evensidemargin=\oddsidemargin

\advance\textheight -70truemm\relax
\advance\topmargin -1truein\relax
\unitlength\textwidth
\divide\unitlength by 150\relax

\usepackage{amsmath,amssymb}
\usepackage{mathptmx}
\usepackage{bm}
\bmdefine{\aaa}{a}
\bmdefine{\bbb}{b}
\bmdefine{\ccc}{c}
\bmdefine{\ddd}{d}
\bmdefine{\eee}{e}
\bmdefine{\sss}{s}
\bmdefine{\uuu}{u}
\bmdefine{\vvv}{v}
\bmdefine{\www}{w}
\bmdefine{\xxx}{x}
\bmdefine{\yyy}{y}
\bmdefine{\zzz}{z}
\bmdefine{\zerovec}{0}

\newcommand{\CCC}{\mathbb{C}}
\newcommand{\RRR}{\mathbb{R}}

\newcommand{\XXX}{\mathbb{X}}

\newcommand{\diag}{{\mathrm{Diag}}}

\newcommand{\qed}{\nolinebreak\rule{.3em}{.6em}}
\newcommand{\rank}{\mathrm{rank}}
\newcommand{\mtrank}{\mathrm{mtrank}\,}
\newcommand{\grank}{\mathrm{grank}\,}
\newcommand{\GL}{{\mathrm{GL}}}

\newcommand{\typicalrankR}{{\mathrm{typical\_rank_\RRR}}}

\numberwithin{equation}{section}
\newtheorem{thm}[equation]{Theorem}

\newtheorem{example}[equation]{Example}
\newtheorem{lemma}[equation]{Lemma}
\newtheorem{cor}[equation]{Corollary}
\newtheorem{definition}[equation]{Definition}
\newtheorem{prop}[equation]{Proposition}

\begin{document}

\title{Typical rank of $m\times n\times (m-1)n$ tensors with
$3\leq m\leq n$ over the real number field}
\author{Toshio Sumi, Mitsuhiro Miyazaki, and Toshio Sakata}

\maketitle

\begin{abstract}
Tensor type data are used 
recently in various application fields, 
and then a typical rank is important.  
Let $3\leq m\leq n$.
We study typical ranks of $m\times n\times (m-1)n$ tensors
over the real number field.
Let $\rho$ be the Hurwitz-Radon function defined as
$\rho(n)=2^b+8c$ for nonnegative integers $a,b,c$ 
such that $n=(2a+1)2^{b+4c}$ and $0\leq b<4$.
If $m \leq \rho(n)$, then 
the set of $m\times n\times (m-1)n$ tensors has two
typical ranks $(m-1)n,(m-1)n+1$.
In this paper, we show that the converse is also true:
if $m > \rho(n)$, then 
the set of $m\times n\times (m-1)n$ tensors has only one
typical rank $(m-1)n$.
\end{abstract}

\section{Introduction}

An analysis of high dimensional arrays is getting frequently used.
Kolda and Bader \cite{Kolda-Bader:2009} introduced many applications 
of tensor decomposition analysis 
in various fields such as signal processing, computer vision, data mining,
and others.

In this paper we concentrate to discuss $3$-way arrays.
A $3$-way array 
$$(a_{ijk})_{1\leq i\leq m,\ 1\leq j\leq n,\ 1\leq k\leq p}$$
with size $(m,n,p)$ is called an $m\times n\times p$
tensor.
A rank of a tensor $T$, denoted by $\rank\, T$, 
is defined as the minimal number of
rank one tensors which describe $T$ as a sum.
The rank depends on the base field.  For example there is
a $2\times 2\times 2$ tensor over the real number field whose rank
is $3$ but is $2$ as a tensor over the complex number field.
\par

Throughout this paper, we assume that the base field is the real number field $\RRR$.
Let $\mathbb{R}^{m\times n\times p}$ be the set of $m\times n\times p$ tensors with Euclidean topology.
A number $r$ is a typical rank of $m\times n\times p$ tensors
if the set of tensors with rank $r$ contains 
a nonempty open semi-algebraic set
of $\mathbb{R}^{m\times n\times p}$
 (see Theorem~\ref{thm:Friedland}).
We denote by $\typicalrankR(m,n,p)$ the set of typical ranks of
$\mathbb{R}^{m\times n\times p}$.
If $s$ (resp. $t$) is the minimal (resp. maximal) number of 
$\typicalrankR(m,n,p)$, then 
$$\typicalrankR(m,n,p)=[s,t],$$
the interval of all integers between $s$ and $t$, including both, and $s$ is equal to the generic rank
of the set of $m\times n\times p$ tensors over the complex number
field \cite{Friedland:2008}.
In the case where $m=2$, the set of typical ranks of $2\times n\times p$ tensor
is well-known \cite{tenBerge-etal:1999}:
$$\typicalrankR(2,n,p)=\begin{cases}
\{p\}, & n<p\leq 2n \\
\{2n\}, & 2n<p \\
\{p,p+1\}, & n=p\geq 2
\end{cases}
$$
Suppose that $3\leq m\leq n$.
If $p>(m-1)n$ then the set of typical ranks of $m\times n\times p$ tensors
is just $\{\min(p,mn)\}$.
If $p=(m-1)n$ then
the set of typical ranks of $m\times n\times p$ tensor 
is $\{p\}$ or $\{p,p+1\}$ \cite{tenBerge:2000}.
Until our paper \cite{Sumi-etal:2010a}, only a few cases where
$\typicalrankR(m,n,(m-1)n)=\{(m-1)n,(m-1)n+1\}$ \cite{Comon-etal:2009,Friedland:2008} are known and
we constructed infinitely many examples 
by using the concept of absolutely nonsingular tensors in \cite{Sumi-etal:2010a}:
If $m\leq \rho(n)$ then $\typicalrankR(m,n,p)=\{p,p+1\}$,
where $\rho(n)$ is the Hurwitz-Radon number given by
$\rho(n)=2^b+8c$ for nonnegative integers $a,b,c$ 
such that $n=(2a+1)2^{b+4c}$ and $0\leq b<4$.

The purpose of this paper is to completely determine the set of 
typical ranks of $m\times n\times (m-1)n$ tensors:

\begin{thm} \label{thm:main}
Let $3\leq m\leq n$ and $p=(m-1)n$.
Then it holds
$$\typicalrankR(m,n,p)=\begin{cases}
\{p\}, & m>\rho(n) \\
\{p,p+1\}, & m\leq \rho(n).
\end{cases}$$
\end{thm}

We denote an $m_1\times m_2\times m_3$ tensor $(x_{ijk})$ by
$(X_1;\ldots; X_{m_3})$, where $X_t=(x_{ijt})$ is an $m_1\times m_2$
matrix for each $1\leq t\leq m_3$.
Let $3\leq m\leq n$ and $p=(m-1)n$.
For an $n\times p\times m$ tensor $X=(X_1;\ldots;X_{m-1};X_{m})$,
let $H(X)$ and $\hat{H}(X)$ be a $p\times p$ matrix and an $mn\times p$ matrix
respectively defined as follows. 
$$H(X)=\begin{pmatrix} X_1\\ X_2\\ \vdots\\ X_{m-1}\end{pmatrix}, \quad
\hat{H}(X)=\begin{pmatrix} X_1\\ X_{2}\\ \vdots\\ X_m\end{pmatrix}$$
Let
$$\mathfrak{R}=\{ X\in \RRR^{n\times p\times m} \mid \text{$H(X)$ is nonsingular}\}.$$
This is a nonempty Zariski open set.
For $X=(X_1;\ldots;X_{m-1};X_m)\in \mathfrak{R}$, we see
$$\hat{H}(X)H(X)^{-1}
=\begin{pmatrix} E_n \\ & E_n \\ &&\ddots \\ &&& E_n \\ Y_1 & Y_2 & \cdots & Y_{m-1} \end{pmatrix},
$$
where $(Y_1,Y_2,\ldots,Y_{m-1})=X_mH(X)^{-1}$.
Note that $\rank\, X\geq p$ for $X\in \mathfrak{R}$.
Let $h$ be an isomorphism from the set of $n\times p$ matrices to $\RRR^{n\times n\times (m-1)}$
given by
$$(Y_1,Y_2,\ldots,Y_{m-1}) \mapsto (Y_1;Y_2;\ldots;Y_{m-1}).$$
Then $h(X_mH(X)^{-1}) \in \RRR^{n\times n\times (m-1)}$.
We consider the following subsets of $\RRR^{n\times n\times (m-1)}$.
For $Y=(Y_1;Y_2;\ldots;Y_{m-1})\in \RRR^{n\times n\times (m-1)}$ and
$\aaa=(a_1,\ldots,a_{m-1},a_m)^\top \in \RRR^{m}$,
let
$$M(\aaa,Y)=\sum_{k=1}^{m-1} a_kY_k-a_mE_n$$
and set
$$\mathfrak{C}=\{Y \in \mathbb{R}^{n\times n\times(m-1)}\mid
|M(\aaa,Y)|<0
\text{ for some $\aaa\in \mathbb{R}^m$}\}$$
and
$$\mathfrak{A}=\{Y \in \mathbb{R}^{n\times n\times(m-1)} \mid
|M(\aaa,Y)|>0 \text{ for all $\aaa\ne \zerovec$}
\}.$$
The subsets $\mathfrak{C}$ and $\mathfrak{A}$ are open sets in Euclidean topology
and $\overline{\mathfrak{C}}\cup \overline{\mathfrak{A}}=\RRR^{n\times n\times(m-1)}$.
In \cite{Sumi-etal:2010a}, we show that $\mathfrak{A}$ is not empty if and only if
$m\leq \rho(n)$ and that $\rank\, X>p$ for any $X\in \mathfrak{R}$
with $h(X_mH(X)^{-1})\in \mathfrak{A}$.
In this paper, we show that there exists an open subset $\mathfrak{F}$ of $\mathfrak{C}$ such that
$\overline{\mathfrak{F}}=\overline{\mathfrak{C}}$ and $\rank\, X=p$ 
for any $X\in \mathfrak{R}$ with $h(X_mH(X)^{-1})\in \mathfrak{F}$.

%%%%%%%%%%%%%%%%%%%%%%%%%%%%%%%%%%%%%%%%%%%%%%%%%%
\section{Typical rank}

Due to \cite{Strassen:1983,tenBerge:2000} and others, 
a number $r$ is a typical rank of tensors of $\RRR^{m_1\times m_2\times m_3}$ if the subset of tensors of $\RRR^{m_1\times m_2\times m_3}$
of rank $r$ has nonzero volume.
In this paper, we adopt the algebraic definition due to Friedland.
These definitions are equivalent, since for any $r\geq 0$, the set of tensors of rank $r$ is a semi-algebraic set
by the Tarski-Seidenberg principle (cf. \cite{Bochnak-Coste-Roy:1998}).

For $\xxx=(x_1,\ldots,x_{m_1})^\top\in \CCC^{m_1}$, 
$\yyy=(y_1,\ldots,y_{m_2})^\top\in \CCC^{m_2}$, and
$\zzz=(z_1,\ldots,z_{m_3})^\top\in \CCC^{m_3}$, we denote $(x_iy_jz_k)\in \CCC^{m_1\times m_2\times m_3}$ by $\xxx \otimes \yyy\otimes \zzz$.
Let $f_t\colon (\CCC^{m_1}\times \CCC^{m_2}\times \CCC^{m_3})^t \to 
\CCC^{m_1\times m_2\times m_3}$ be a map given by
$$f_t(\xxx_{1,1}, \xxx_{1,2}, \xxx_{1,3}, \ldots, \xxx_{t,1}, \xxx_{t,2}, \xxx_{t,3}) =
\sum_{\ell=1}^t \xxx_{\ell,1} \otimes \xxx_{\ell,2} \otimes \xxx_{\ell,3}.$$
Let $S$ be a subset of $\RRR^{m_1\times m_2\times m_3}$.
$S$ is called semi-algebraic if
it is a finite Boolean combination (that is, a finite composition of disjunctions, conjunctions and negatios) of sets of the form 
\begin{equation} \label{eq:>0}
\{(a_{ijk})\in \RRR^{m_1\times m_2\times m_3} \mid f(a_{111},\ldots,a_{m_1,m_2,m_3})>0\}
\end{equation}
and
$$\{(a_{ijk})\in \RRR^{m_1\times m_2\times m_3} \mid g(a_{111},\ldots,a_{m_1,m_2,m_3})=0\},$$ 
where $f$ and $g$ are polynomials in $m_1m_2m_3$ indeterminates $x_{111} ,\ldots, x_{m_1,m_2,m_3}$ over $\RRR$.
Then $S$ is an open semi-algebraic set if and only if it is expressed as a finite Boolean combinations
of sets of the form \eqref{eq:>0}, and it is a dense open semi-albebraic set if and only if it is a
Zariski open set, that is, expressed as
$$\{(a_{ijk})\in \RRR^{m_1\times m_2\times m_3} \mid g(a_{111},\ldots,a_{m_1,m_2,m_3})\ne0\}.$$

\begin{thm}[{\cite[Theorem~7.1]{Friedland:2008}}] \label{thm:Friedland}
The space $\RRR^{m_1\times m_2\times m_3}$, $m_1,m_2,m_3 \in\mathbb{N}$, 
contains a finite number of open connected disjoint semi-algebraic sets 
$O_1,\ldots,O_M$ satisfying the following properties.
\begin{enumerate}
\item $\RRR^{m_1\times m_2\times m_3}\smallsetminus \cup_{i=1}^M O_i$
is a closed semi-algebraic set $\RRR^{m_1\times m_2\times m_3}$ 
of dimension less than $m_1m_2m_3$.
\item  Each $T \in O_i$ has rank $r_i$ for $i = 1,\ldots,M$.
\item The number $\min(r_1,\ldots,r_M)$ is equal to the generic rank $\grank(m_1,m_2,m_3)$ of
$\CCC^{m_1\times m_2\times m_3}$, that is, the minimal $t\in \mathbb{N}$
such that the closure of the image of $f_t$ is equal to $\CCC^{m_1\times m_2\times m_3}$.
\item $\mtrank(m_1,m_2,m_3):=\max(r_1,\ldots,r_M)$ is the minimal $t\in \mathbb{N}$ such that the closure of $f_t((\RRR^{m_1}\times \RRR^{m_2}\times \RRR^{m_3})^k)$ is equal to $\RRR^{m_1\times m_2\times m_3}$.
\item For each integer $r\in [\grank(m_1,m_2,m_3),\mtrank(m_1,m_2,m_3)]$,
there exists $r_i=r$ for some integer $i\in [1,M]$.
\end{enumerate}
\end{thm}

\begin{definition}\rm 
A positive number $r$ is called a typical rank of $\RRR^{m_1\times m_2\times m_3}$ 
if 
$$r \in [\grank(m_1,m_2,m_3),\mtrank(m_1,m_2,m_3)].$$
Put 
$$\typicalrankR(m_1,m_2,m_3)=[\grank(m_1,m_2,m_3),\mtrank(m_1,m_2,m_3)].$$
\end{definition}

We state basic facts.

\begin{prop} \label{prop:open}
Let $r$ be a positive number and $U$ a nonempty open set of\/ $\RRR^{m_1\times m_2\times m_3}$.
If every tensor of $U$ has rank $r$, then $r$ is a typical rank of\/ $\RRR^{m_1\times m_2\times m_3}$. 
\end{prop}
\begin{proof}
Let $O_1,\ldots,O_M$ be open connected disjoint semi-algebraic sets
as in Theorem~\ref{thm:Friedland}.
Since $\dim(\RRR^{m_1\times m_2\times m_3}\smallsetminus \cup_{i=1}^M O_i)<m_1m_2m_3$, there exists $i\in [1,M]$ such that $U\cap O_i$ is not empty.
\end{proof}

\begin{prop} \label{prop:compare}
Let $m_1,m_2,m_3,m_4 \in \mathbb{N}$ with $m_3 < m_4$. Then
$$\grank(m_1,m_2,m_3) \leq  \grank(m_1,m_2,m_4)$$
and
$$\mtrank(m_1,m_2,m_3) \leq \mtrank(m_1,m_2,m_4).$$
\end{prop}

\begin{proof}
Let $U$ be the nonempty Zariski open subset $U$ of $\CCC^{m_1\times m_2\times m_4}$ 
consisting of all tensors of rank $\grank(m_1,m_2,m_4)$ 
and put
$$V=\{(Y_1;Y_2;\ldots;Y_{m_3}) \in \CCC^{m_1\times m_2\times m_3} \mid (Y_1;Y_2;\ldots;Y_{m_4}) \in U\}.$$
Then $V$ is a nonempty Zariski open set of $\CCC^{m_1\times m_2\times m_3}$. 
For the subset $U^\prime$ of $\CCC^{m_1\times m_2\times m_3}$ consisting of all tensors of rank $\grank(m_1,m_2,m_3)$, the intersection $V\cap U^\prime$ is a nonempty Zariski open set.
Since $\rank\, Y \leq \rank(Y;X)$ for $Y \in \CCC^{m_1\times m_2\times m_3}$ and
$(Y;X) \in \CCC^{m_1\times m_2\times m_4}$, we see 
$$\grank(m_1,m_2,m_3) \leq \grank(m_1,m_2,m_4).$$
\par
Next, take an open semi-algebraic set $V$ of $\RRR^{m_1\times m_2\times m_3}$ consisting of tensors
of rank $\mtrank(m_1,m_2,m_3)$. Then there are $s \in \typicalrankR(m_1,m_2,m_4)$ and an open semi-algebraic set $O$ of $\RRR^{m_1\times m_2\times m_4}$ consisting of tensors of rank $s$ such that 
$\{(A;B) | A \in V, B \in \RRR^{m_1\times m_2\times (m_4-m_3)}\}\cap O \ne \varnothing$. Thus
$$\mtrank(m_1,m_2,m_3) \leq s \leq \mtrank(m_1,m_2,m_4).$$
\qed
\end{proof}

The action of $\GL(m)\times \GL(n)\times \GL(p)$ on $\RRR^{m\times n\times p}$
is given as follows.
Let $P=(p_{ij})\in \GL(n)$, $Q=(q_{ij}) \in \GL(m)$, and
$R=(r_{ij}) \in \GL(p)$.
The tensor $(b_{ijk})=(P,Q,R)\cdot (a_{ijk})$ is defined as
$$b_{ijk}=\sum_{s=1}^m\sum_{t=1}^n\sum_{u=1}^p p_{is}q_{jt}r_{ku}a_{stu}.$$
Therefore, 
$$(P,Q,R)\cdot (A_1;\ldots;A_p)
=(\sum_{u=1}^p r_{1u}PA_uQ^\top;\ldots;\sum_{u=1}^p r_{pu}PA_uQ^\top).$$

\begin{definition}\rm
Two tensors $A$ and $B$ is called {\sl equivalent}
if there exists $g\in \GL(m)\times \GL(n)\times \GL(p)$ such that
$B=g\cdot A$.
\end{definition}
%%%

\begin{prop}
If two tensors are equivalent, then they have the same rank.
\end{prop}

A $1\times m_2\times m_3$ tensor $T$ is an $m_2\times m_3$ matrix and 
$\rank\, T$ is equal to the matrix rank.
The following three propositions are well-known.

\begin{prop}
Let $m_1,m_2,m_3 \in \mathbb{N}$ with $2 \leq m_1 \leq m_2 \leq m_3$. 
If $m_1m_2 \leq m_3$, then typical
rank of\/ $\RRR^{m_1\times m_2\times m_3}$ is only one integer $m_1m_2$.
\end{prop}

\begin{prop} \label{prop:diagonaldecomp}
An $m_1\times m_2\times m_3$ tensor $(Y_1;\ldots;Y_{m_3})$ has rank 
less than or equal to $r$ if and only if
there are an $m_1\times r$ matrix $P$, an $r\times m_2$ matrix $Q$,
and $r\times r$ diagonal matrices $D_1,\ldots,D_{m_3}$
such that
$Y_k=PD_kQ$
for $1\leq k\leq m_3$.
\end{prop}

\begin{prop}
Let $X=(x_{ijk})$ be an $m_1\times m_2\times m_3$ tensor.
For an $m_2\times m_1\times m_3$ tensor $Y=(x_{jik})$ and
an $m_1\times m_3\times m_2$ tensor $Z=(x_{ikj})$, it holds that
$$\rank\, X=\rank\, Y=\rank\, Z.$$
\end{prop}

For an integer $2\leq m<n<2m$, the number $n$ is an only typical rank of $\RRR^{m\times n\times 2}$.  Indeed, it is known that
\begin{thm}[\cite{Miyazaki-etal:2009}]
Let $2\leq m<n$.
There is an open dense semi-algebraic set $O$ of $\RRR^{m\times n\times 2}$ 
of which any tensor is equivalent to $((E_{m},O);(O,E_{m}))$ which has rank $\min(n,2m)$.
\end{thm}
%%%%
Furthermore, by Proposition~\ref{prop:compare}, 
$\typicalrankR(m,m,2)$ is equal to either $\{m\}$ or $\{m,m+1\}$.
Let $U$ be an open subset of $\RRR^{m\times m\times 2}$ consisting of $(A;B)$ such that $A$ is an $m\times m$ nonsingular matrix 
and all eigenvalues of $A^{-1}B$ are distinct and contain non-real numbers.
For $m\geq 2$, the set $U$ is not empty and any tensor of $U$ has rank $m+1$
(cf. \cite[Theorem~4.6]{Sumi-etal:2009})
and therefore $\typicalrankR(m,m,2)=\{m,m+1\}$ by Proposition~\ref{prop:open}.

\begin{thm}[{\cite[Result 2]{tenBerge:2000}}] \label{thm:tall}
Let $m,n,\ell \in\mathbb{N}$ with $3\leq m \leq n \leq u$. 
If $(m-1)n <u <mn$,
then typical rank of\/ $\RRR^{m\times n\times u}$ is only one integer $u$.
\end{thm}

Ten Berge showed it by applying Fisher's
result \cite[Theorem 5.A.2]{Fisher:1966} for a map defined by
using the Moore-Penrose inverse. 
However the Moore-Penrose inverse is not continuous on the set of matrices
and thus not analytic.
So, until this section, we give another proof for reader's convenience.

Let $3\leq m\leq n$, $p=(m-1)n$, $p<u<mn$ and $q=u-p-1$.
For $W\in M(n-1,n;\RRR)$, the set of $(n-1)\times n$ matrices,
we define a vector 
$W^\perp=(a_1,\ldots,a_{n})^\top$ in $\RRR^{n}$
by
$$a_j=(-1)^{n+j} |W_{[j]}|$$
for $j=1,\ldots,n$, where
$W_{[j]}$ is an $(n-1)\times (n-1)$ matrix obtained from $W$
by removing the $j$-th column.

The following properties are easily shown.

\begin{enumerate}
\item $W^\perp=\zerovec$ if and only if $\rank\, W<n-1$.
\item $WW^\perp=\zerovec$.
\end{enumerate}

\medskip

Let $A_k$ be an $n\times u$ matrix for $1\leq k\leq m$.
Let $B_j$ be a $q\times u$ matrix defined by 
$(O_{p+1},E_q)$ for $j\leq p+1$, and
by $(O_{p},\eee_{j-p-1},\diag(E_{j-p-2},0,E_{u-j}))$
for $p+2\leq j\leq u$, where $E_k$ is the $k\times k$ identity matrix and
$\eee_j$ is the $j$-th column of
the identity matrix with suitable size.
Put 
\begin{equation}
\label{eq:XjYj}
X_j=\begin{pmatrix} A_2-jA_1\\ A_3-j^2A_1\\ \vdots\\ A_{m}-j^{m-1}A_1
\end{pmatrix} \text{ and }
Y_j=\begin{pmatrix} X_j\\ B_j\end{pmatrix}
\end{equation}
for $1\leq j\leq u$, and
\begin{equation}
\label{eq:H}
H=(Y_1^\perp,\ldots,Y_{u}^\perp).
\end{equation}
%%%%%%
We define a polynomial $h$ on $\RRR^{n\times u\times m}$ by
$$
h(A_1;A_2;\dots;A_{m})=|H|.
$$

We show that the polynomial $h(A_1;A_2;\ldots;A_m)$ is not zero.
It suffices to show that $h(A_1;A_2;\ldots;A_m)\ne 0$ for some
tensor $(A_1;A_2;\ldots;A_m)$.
We prepare a lemma.

Let $f(a_1,\ldots,a_{m-1},b)=
\left|\begin{matrix}a_1-b&\cdots&a_{m-1}-b\\
a_1^2-b^2&\cdots&a_{m-1}^2-b^2\\ \vdots&&\vdots\\
a_1^{m-1}-b^{m-1}&\cdots&a_{m-1}^{m-1}-b^{m-1}\end{matrix}\right|$.

\begin{lemma} \label{lem:Vandermond}
If $a_1,\ldots,a_{m-1},b$ are distinct eath other, then
$f(a_1,\ldots,a_{m-1},b)\ne 0$.
\end{lemma}

\begin{proof}
It is easy to see that
\begin{equation*}
\begin{split}
f(a_1,\ldots,a_{m},b)&=
\left|\begin{matrix}
1&0&\cdots&0\\
b&a_1-b&\cdots&a_{m-1}-b\\
b^2&a_1^2-b^2&\cdots&a_{m-1}^2-b^2\\ 
\vdots&\vdots&&\vdots\\
b^{m-1}&a_1^{m-1}-b^{m-1}&\cdots&a_{m-1}^{m-1}-b^{m-1}\end{matrix}\right| \\[2mm]
&=
\left|\begin{matrix}
1&1&\cdots&1\\
b&a_1&\cdots&a_{m-1}\\
b^2&a_1^2&\cdots&a_{m-1}^2\\ 
\vdots&\vdots&&\vdots\\
b^{m-1}&a_1^{m-1}&\cdots&a_{m-1}^{m-1}\end{matrix}\right| \ne 0.\\
\end{split}
\end{equation*}
\qed
\end{proof}

\begin{lemma}
Let $\vvv=(1,\ldots,1)^T\in \RRR^{n}$,
$A_1=(E_{n},\ldots,E_{n},\vvv,O_q)$ and
$$A_{s+1}=(A_1\eee_1,2^sA_1\eee_2,\ldots,u^sA_1\eee_{u})
=A_1\diag(1^s,2^s,\ldots,u^s)$$
for $1\leq s\leq m-1$.
Then the $(u-1)\times u$ matrix $Y_j$ defined in \eqref{eq:XjYj} satisfies
that $Y_j^\perp=t_j\eee_j$ for some $t_j\ne 0$.
In particular, $h(A_1;A_2;\ldots;A_{m})\ne 0$.
\end{lemma}

\begin{proof}
Let 
$$D_{t,s,j}=\diag(((t-1)n+1)^s-j^s,((t-1)n+2)^s-j^s,\ldots,(tn)^s-j^s)$$
be an $n\times n$ matrix.
Then
$$
A_{s+1}-j^sA_1=(D_{1,s,j},D_{2,s,j},\ldots,D_{m-1,s,j},((p+1)^s-j^s)\vvv,O_q).
$$
For a $v\times w$ matrix $G=(g_{ij})$, we denote
by 
$$G_{=\{a_1,\ldots,a_c\}}^{=\{b_1,\ldots,b_r\}}$$
the $r\times c$ matrix obtained from $G$
by choosing $a_1$-, $\ldots$, $a_c$-th columns 
and $b_1$-, $\ldots$, $b_r$-th rows,
that is 
$(g_{b_ia_j})$,
and put
$$G_{=\{a_1,\ldots,a_c\}}=G_{=\{a_1,\ldots,a_c\}}^{=\{1,\ldots,v\}}, \quad
G_{\leq c}=G_{=\{1,\ldots,c\}}^{=\{1,\ldots,v\}},\quad
G_{\leq c}^{\leq r}=G_{=\{1,\ldots,c\}}^{=\{1,\ldots,r\}}.
$$

First we suppose that $j>p$.
Put $S_t=\{t,n+t,2n+t,\ldots,(m-2)n+t\}$ and 
$M_{j,t}=(Y_j)_{=S_t}^{=S_t}=(X_j)_{=S_t}^{=S_t}$.
Note that $M_{j,t}$ is nonsingular by Lemma~\ref{lem:Vandermond},
since 
$$
|M_{j,t}|=f(t,n+t,2n+t,\ldots,(m-2)n+t,j).
$$
We consider the $p\times p$ matrix 
$(Y_j)_{\leq p}^{\leq p}=(X_j)_{\leq p}$.
There exists a permutation matrix $P$ such that
$$
P^{-1}(X_j)_{\leq p}P
=\diag(M_{j,1},M_{j,2},\ldots,M_{j,n}).
$$
Thus we get
$$
|(X_j)_{\leq p}|
  =\prod_{1\leq t\leq m-1} |M_{j,t}|
$$
which implies that $(X_j)_{\leq p}$ is nonsingular.
Thus $\rank\, Y_j=u-1$ and $Y_j^\perp=t_j\eee_j$ for some $t_j\ne0$,
since the $j$-th column vector of $Y_j$ is zero.
\par
%%%%%%%
Next suppose that $j\leq p$.
The $j$-th column of $Y_j$ is zero.
Let 
$$Z_j=(X_j)_{=\{1,\ldots,p+1\}\smallsetminus\{j\}}$$
be the $p\times p$ matrix obtain from 
$(X_j)_{\leq p+1}$ by removing the $j$-th column.
It suffices to show that $\rank\, Z_j=p$.
We express $j$ uniquely by $ns_0+t_0$ for a pair $(s_0,t_0)$ of integers 
with $0\leq s_0\leq m-2$ and $1\leq t_0\leq n$.
Let 
$$T=\{sn+t_0\mid 0\leq s\leq m-2, s\ne s_0\}\cup\{p+1\}.$$
There exist permutation matrices $P$ and $Q$ such that
$$PZ_jQ
=\begin{pmatrix} 
  \stackrel{\vbox to 4pt{\hbox{\normalsize$\diag$}}}
    {\lower0.5ex\vbox to 0pt{
       \hbox{\scriptsize$\substack{1\leq t\leq n, \\t\ne t_0}$}}} M_{j,t} &
  \begin{matrix} O_{p-m+1,m-2} & *\end{matrix}\\[4mm]
  O_{m-1,p-m+1} & (X_j)_{=T}^{=S_{t_0}} \end{pmatrix}$$
of which last column corresponds to the $(p+1)$-th column of $X_j$.
We get the equality
$$
|Z_j|=(-1)^a|(X_j)_{=T}^{=S_{t_0}}|
\prod_{1\leq t\leq m-1, t\ne t_0} |M_{j,t}|.
$$
Again by Lemma~\ref{lem:Vandermond}, 
$Z_j$ is nonsingular and $Y_j^\perp=t_j\eee_j$ for some $t_j\ne 0$.
\qed
\end{proof}
\medskip

Thus the polynomial $h$ is not zero.
%%%%
Consider a nonempty Zariski open set
$$S=\{(A_1;A_2;\ldots;A_{m}) \in \RRR^{n\times u\times m} \mid 
  h(A_1;A_2;\ldots;A_{m})\ne 0\}.$$
Note that the closure $\overline{S}$ of $S$ is equal to 
$\RRR^{n\times u\times m}$.
For $(A_1;A_2;\ldots;A_{m})\in S$ and $X_j,Y_j, H$ matrices
given in \eqref{eq:XjYj} and \eqref{eq:H},
$A_kY_j^\perp=j^{k-1}A_1Y_j^\perp$ for $1\leq k\leq m$ and 
$1\leq j\leq u$.
Since 
\begin{equation*}
\begin{split}
A_kH &=(A_kY_1^\perp,A_kY_2^\perp,\ldots,A_kY_u^\perp) \\
&=
(A_1Y_1^\perp,2^{k-1}A_1Y_2^\perp,\ldots,u^{k-1}A_1Y_u^\perp) \\
&=
A_1H\diag(1,2^{k-1},\ldots,u^{k-1}),
\end{split}
\end{equation*}
it holds that $A_k=A_1H\diag(1,2^{k-1},\ldots,u^{k-1})H^{-1}$
for each $k$.
By Proposition~\ref{prop:diagonaldecomp}, we get 
$\rank(A_1;A_2;\ldots;A_{m})\leq u$.
Any number of $\typicalrankR(m,u,n)$ is greater than or equal to 
$u$ which is equal to the generic rank of $\CCC^{m\times n\times u}$,
since $(m-1)n<u<mn$.
This completes the proof of Theorem~\ref{thm:tall}.

\begin{cor} \label{cor:mxnx(m-1)n}
Let $3\leq m\leq n$.
Then the set of typical ranks of $m\times n\times (m-1)n$ tensors
is either $\{(m-1)n\}$ or $\{(m-1)n,(m-1)n+1\}$.
\end{cor}

\begin{proof}
The typical rank of $\RRR^{m\times n\times ((m-1)n+1)}$ is only $(m-1)n+1$ 
by Theorem~\ref{thm:tall} and
the minimal typical rank of $\RRR^{m\times n\times (m-1)n}$ is equal to 
$(m-1)n$, since it is equal to
the generic rank of $\CCC^{m\times n\times (m-1)n}$.
Thus the assertion follows from Proposition~\ref{prop:compare}.
\qed
\end{proof}

\section{Characterization}

From now on, let $3\leq m\leq n$, $\ell=m-1$ and $p=(m-1)n$.
For an $n\times n\times \ell$ tensor $(Y_1;\ldots;Y_\ell)$, 
consider an $n\times p\times m$ tensor 
$X(Y_1,\ldots,Y_\ell)=(X_1;\ldots;X_m)$ given by 
\begin{equation} \label{eq:typicalform}
\begin{pmatrix} X_1 \\ \vdots \\ X_m\end{pmatrix}
=\begin{pmatrix} E_n \\ & E_n \\ &&\ddots \\ &&& E_n \\ Y_1 & Y_2 & \cdots & Y_{\ell} \end{pmatrix}.
\end{equation}
Note that $\rank\, X(Y_1,\ldots,Y_\ell)\geq p$, 
since $\rank\, X(Y_1,\ldots,Y_\ell)$ is greater than or equal to
the rank of the $p\times p$ matrix \eqref{eq:typicalform}.
In generic, an $m\times n\times p$ tensor is equivalent to
a tensor of type as $X(Y_1,\ldots,Y_\ell)$.

We denote by $\mathfrak{M}$ the set of
tensors 
$Y=(Y_1;\ldots;Y_\ell)\in \mathbb{R}^{n\times n\times \ell}$ 
such that
there exist an $m\times p$ matrix $(x_{ij})$ and an $n\times p$ matrix 
$A=(\aaa_1,\ldots,\aaa_p)$ such that
\begin{equation} \label{eq:lineq}
(x_{1j}Y_1+\cdots+x_{m-1,j}Y_{m-1}-x_{mj}E_n)\aaa_j=\zerovec
\end{equation}
for $1\leq j\leq p$
and
\begin{equation} \label{eqn:B}
B:=\begin{pmatrix} AD_1\\ \vdots \\ AD_{\ell}\end{pmatrix}
\end{equation}
is nonsingular, where
$D_k=\diag(x_{k1},\cdots,x_{kp})$ for $1\leq k\leq \ell$.

\begin{lemma} \label{lem:equiv}
$\rank\, X(Y_1,\ldots,Y_\ell)=p$ if and only if 
$(Y_1;\ldots;Y_\ell)\in \mathfrak{M}$.
\end{lemma}

\begin{proof}
Suppose that $\rank X(Y_1,\ldots,Y_\ell)=p$.
There are an $n\times p$ matrix $A$, a $p\times p$ matrix $Q$
and $p\times p$ diagonal matrices $D_i$ such that $X_k=AD_kQ$
for $k=1,\ldots,m$.
Since $$\begin{pmatrix} X_1 \\ \vdots \\ X_\ell\end{pmatrix}=E_p=
\begin{pmatrix} AD_1 \\ \vdots \\ AD_\ell\end{pmatrix}Q,$$
$B$ is nonsingular.
Then 
$(Y_1,\ldots,Y_\ell)B=AD_m$ implies that
$\sum_{k=1}^\ell Y_kAD_k=AD_m$.
Therefore, the $j$-th column vector $\aaa_j$ of $A$
satisfies \eqref{eq:lineq}.
Therefore $(Y_1,\ldots,Y_\ell)\in\mathfrak{M}$.
It is easy to see that the converse is also true.
\qed
\end{proof}

For an $n\times n\times \ell$ tensor $Y=(Y_1;\ldots;Y_\ell)$,
we put 
$$V(Y)=\{\aaa \in \RRR^n \mid \sum_{k=1}^\ell x_kY_k\aaa=x_m\aaa
\text{ for some $(x_1,\ldots,x_m)^\top\ne\zerovec$}\}.$$
The set $V(Y)$ is not a vector subspace of $\RRR^n$.
Let $\hat{V}(Y)$ be the smallest vector subspace of $\RRR^n$ including $V(Y)$.
Let 
$$\mathfrak{S}=\{Y \in \RRR^{n\times n\times \ell} \mid \dim \hat{V}(Y)=n\}.$$ 

\begin{prop} \label{prop:UsubsetS}
$\mathfrak{M}\subset\mathfrak{S}$ holds.
\end{prop}

\begin{proof}
Let $Y\in \mathfrak{M}$.
Consider the matrix $B$ in \eqref{eqn:B} for any 
$m\times p$ matrix $(x_{ij})$ and any $n\times p$ matrix $A=(\aaa_1,\ldots,\aaa_p)$ satisfying the equation \eqref{eq:lineq}.
By column operations, $B$ is transformed to
a $p\times p$ matrix having a form
$$\begin{pmatrix}
P_{11} & O_{n,p-\dim \hat{V}(Y)} \\
P_{21} & P_{22}
\end{pmatrix}$$
where $P_{11}$ is an $n\times \dim \hat{V}(Y)$ submatrix of $A$.
Since $B$ is nonsingular, $P_{11}$ is also nonsingular, which
implies that $\dim \hat{V}(Y)=n$.
\qed
\end{proof}

By Corollary~\ref{cor:mxnx(m-1)n}, Lemma~\ref{lem:equiv} and Proposition~\ref{prop:UsubsetS}, 
we have the following

\begin{prop} \label{prop:rankX(Y)}
If $\rank\, X(Y)=p$ then $Y\in \mathfrak{S}$.
In particular, 
$\overline{\mathfrak{S}}\ne \RRR^{n\times n\times\ell}$
implies that $\typicalrankR(m,n,p)=\{p,p+1\}$.
\end{prop}

\begin{thm}[\cite{Sumi-etal:2010a}] \label{thm:ans}
If $(Y_1;\ldots;Y_\ell;E_n)$ is an absolutely nonsingular tensor, then
it holds that
$\rank\, X(Y_1,\ldots,Y_\ell)>p$.
\end{thm}

Here $(Y_1;\ldots;Y_\ell;Y_m)$ is called an absolutely nonsingular
tensor
if $|\sum_{k=1}^m x_kY_k|=0$ implies $(x_1,\ldots,x_m)^\top=\zerovec$.
Therefore,

\begin{prop} \label{prop:dimV(Y)=0}
$\dim \hat{V}(Y)=0$ if and only if 
$(Y;E_n)$ is an $n\times n\times m$ absolutely nonsingular tensor.
\end{prop}

Note that there exists an $n\times n\times m$ absolutely
nonsingular tensor if and only if $m$ is less than or equal to 
the Hurwitz-Radon number $\rho(n)$ \cite{Sumi-etal:2010a}.

\begin{prop}
Let $Y$ and $Z$ be $n\times n\times m$ tensors.
Suppose $(P,Q,R)\cdot Y=Z$ for $(P,Q,R)\in \GL(n)\times \GL(n)\times \GL(m)$.
Then $V(Y)=Q^\top V(Z)=\{Q^\top\yyy \mid \yyy\in V(Z)\}$.
In particular, $\dim \hat{V}(Z)=\dim \hat{V}(Y)$.
\end{prop}

\begin{proof}
Suppose that $\sum_{k=1}^m x_kZ_k\yyy=\zerovec$.
Then from the definition of the action, it follows that
$$\sum_{k=1}^m d_k\sum_{u=1}^m r_{ku}PY_uQ^\top\yyy=
P(\sum_{u=1}^m (\sum_{k=1}^m d_kr_{ku}Y_u))Q^\top\yyy=\zerovec.$$
Thus $Q^\top\yyy \in V(Y)$.
\qed
\end{proof}

%%%%%%%%%%%%%%%%%%%%%%%%%%%%%
%%%%%%%%%%%%%%%%%%%%%%%%%%%%%
%%%%%%%%%%%%%%%%%%%%%%%%%%%%%

\begin{cor}
$\mathfrak{S}$ is closed under the equivalence relation.
\end{cor}

The closure of the set of all $n\times p\times m$ tensors equivalent to
$X(Y_1,\ldots,Y_\ell)$ for some $Y_1,\ldots,Y_{\ell}$ 
is $\RRR^{n\times p\times m}$.
Furthermore, the following claim holds.
Let $\mathfrak{V}$ be the set of $n\times p\times m$ tensors
$(X_1;\ldots;X_m)$
such that
$A=(X_1^\top,\ldots,X_\ell^\top)$
is a nonsingular $p\times p$ matrix
and $(Y_1;\ldots;Y_\ell)$ given by
$(Y_1,\ldots,Y_\ell)=A^{-1}X_m$ lies in $\mathfrak{M}$.
Any tensor of $\mathfrak{V}$ has rank $p$.
If $\mathfrak{M}$ is dense in $\mathbb{R}^{n\times n\times \ell}$
then $\mathfrak{V}$ is dense in $\mathbb{R}^{n\times p\times m}$.

\section{Classes of $n\times n\times \ell$ tensors}

We separate $\RRR^{n\times n\times \ell}$ into three classes
$\mathfrak{A}$, $\mathfrak{C}$, and $\mathfrak{B}$ as follows.
Let $\mathfrak{A}$ be the set of tensors $Y$ such that
$(Y;E_n)$ is absolutely nonsingular.
By Proposition~\ref{prop:dimV(Y)=0},
we have the following

\begin{prop} \label{prop:C1capU}
$\mathfrak{A}\cap \mathfrak{S}=\varnothing$.
\end{prop}

From now on, we use symbols $x_1,\ldots,x_\ell,x_m$ as indeterminates over $\RRR$.
For $Y=(Y_1;\ldots;Y_\ell)\in \mathbb{R}^{n\times n\times\ell}$, we define the
$n\times n$ matrix with entries in $\RRR[x_1,\ldots,x_\ell,x_m]$ as follows.
$$M(\xxx,Y)=\sum_{k=1}^\ell x_kY_k-x_mE_n$$
Note that fixing $a_1,\ldots,a_\ell$, the determinant
$|M(\aaa,Y)|$ is positive for $a_m \ll 0$, where $\aaa=(a_1,\ldots,a_\ell,a_m)^\top$.
Set
$$\mathfrak{C}=\{Y \in \mathbb{R}^{n\times n\times\ell}\mid
|M(\aaa,Y)|<0
\text{ for some $\aaa\in \mathbb{R}^m$}\}.$$
Note that $\mathfrak{C}$ is not empty, and 
if $n$ is not congruent to $0$ modulo $4$
then $\mathfrak{A}$ is empty since $m\geq 3$.
Set $\mathfrak{B}=\RRR^{n\times n\times \ell}\smallsetminus
(\mathfrak{A}\cup \mathfrak{C})$.
The class $\mathfrak{B}$ contains the zero tensor.

\begin{prop}
$\mathfrak{A}$ and $\mathfrak{C}$ are open subsets of\/ $\RRR^{n\times n\times \ell}$.
\end{prop}

Recall that 
$$\mathfrak{A}=\{Y \in \mathbb{R}^{n\times n\times\ell} \mid
|M(\aaa,Y)|>0 \text{ for all $\aaa\ne \zerovec$}
\}.$$
Thus it holds
$$
\mathfrak{B}=\{Y \in \mathbb{R}^{n\times n\times\ell} \mid 
\begin{array}{l}
|M(\bbb,Y)|=0 \text{ for some $\bbb\ne \zerovec$ and} \\
 |M(\aaa,Y)|\geq 0 \text{ for all $\aaa$}
\end{array}
\}.
$$

\begin{prop} \label{prop:C3isboundaryC2}
$\mathfrak{B}$ is a boundary of $\mathfrak{C}$.
In particular, $\mathbb{R}^{n\times n\times \ell}$ is a disjoint sum
of $\mathfrak{A}$ and the closure $\overline{\mathfrak{C}}$ of $\mathfrak{C}$.
\end{prop}

\begin{proof}
It suffices to show that 
$\mathfrak{B} \subset \overline{\mathfrak{C}}$.
Let $Y=(Y_1;\ldots;Y_\ell) \in \mathfrak{B}$.
There are a nonzero vector $\bbb=(b_1,\ldots,b_\ell,b_m)^\top\in\mathbb{R}^n$ 
with $|M(\bbb,Y)|=0$ and an element $g \in \GL(\ell)$
such that $g\cdot Y=(Z_1;Z_2;\ldots;Z_\ell)$ and
$Z_1=\sum_{k=1}^\ell b_kY_k$.
Then $|Z_1-b_mE_n|=0$.
Take a sequence $\{Z_1^{(u)}\}_{u\geq 1}$ such that
$|Z_1^{(u)}-b_mE_n|<0$ and $\lim_{u\to\infty} Z_1^{(u)}=Z_1$.
Thus, $(Z_1^{(u)};Z_2;\ldots;Z_\ell)\in \mathfrak{C}$ and then
$g^{-1}\cdot (Z_1^{(u)};Z_2;\ldots;Z_\ell)\in \mathfrak{C}$.
Therefore, $Y\in \overline{\mathfrak{C}}$.
%%%%%%%
%%%%%
\qed
\end{proof}

\begin{cor}
If $\mathfrak{A}$ is not empty then
$\mathfrak{B}$ is a boundary of $\mathfrak{A}$.
\end{cor}

The set $\mathfrak{B}$ contains a nonzero tensor in general.
We give an example.

\begin{example}\rm 
Let $A=(A_1;A_2;A_3)$ be a $6\times 6\times 3$ tensor
given by
$$X(x_1,x_2,x_3)=x_1A_1+x_2A_2-x_3A_3=\begin{pmatrix}
-x_3&-x_2&0&0&0&-x_1\\
x_1&-x_3&x_2&0&0&0 \\
0&x_1&-x_3&x_2&0&0 \\
0&0&x_1&-x_3&-x_2&0\\
0&0&0&x_1&-x_3&x_2\\
-x_2&0&0&0&x_1&-x_3
\end{pmatrix}
$$
Then $|a_1A_1+a_2A_2-a_3A_3|=a_3^2(a_1a_2-a_3^2)^2+(a_1^3+a_2^3)^2\geq 0$.
The equality holds if $a_3=0$ and $a_1=-a_2$.
Thus $\dim \hat{V}((A_1;A_2))=1$.
Let $B=\begin{pmatrix} 1&\cdots&1\\ 0 &\cdots&0 \\ \vdots&&\vdots\\
0&\cdots&0\end{pmatrix}$ be a $6\times 6$ matrix.
If $x_3=y$, $x_1=-y^2$, and $x_2=-2y/5$, then
$$|X+yB|=y^6(y^6+y^5-7y^4/5+161y^3/125-167y^2/125+629y/625-2926/15625).$$ 
Thus, if $|a_3|$ is sufficiently small 
then
$|X(-a_3^2,-2a_3/5,a_3)+a_3B|<0$.
%%%%%
%%%%%
\end{example}

\begin{prop} \label{prop:notFull}
If $m\leq \rho(n-1)$ then $\mathfrak{C} \not\subset \mathfrak{S}$,
where $\rho(n-1)$ is a Hurwitz-Radon number.
\end{prop}

\begin{proof}
Let $(A_1;\ldots;A_\ell;E_{n-1})$ be an $(n-1)\times (n-1)\times m$ absolutely nonsingular tensor.
Put $B_k=\diag(a_k,A_k)$ for $1\leq k\leq \ell$ and $B_m=\diag(1,E_{n-1})=E_n$, and
$B=(B_1;\ldots;B_\ell)$.
Then it is easy to see that
$B \in \mathfrak{C}$ and $|\sum_{k=1}^\ell x_kB_k-zB_m|=0$ implies 
$z=\sum_{k=1}^\ell a_kx_k$.
Therefore $V(B)=\{a(1,0,\ldots,0)^\top \in \RRR^n \mid a\in \RRR\}$.
In particular $B \notin \mathfrak{S}$.
\qed
\end{proof}

\section{Irreducibility}
In the space of homogeneous polynomials in $m$ variables, there exists  
a proper Zariski closed subset $S$ such that 
if a polynomial does not belong to $S$ then it is irreducible \cite[Theorem~7]{Kaltofen:1995}, since $m\geq 3$.
Let $P(m,n)$ be the set of
homogeneous polynomials in $m$ variables $x_1,\ldots,x_m$ with real coefficients of degree $n$ such that the coefficient of $x_m^n$ is one.
Its dimension is $\binom{m+n-1}{m-1}-1$.
Let $I_\ell$ be a nonempty Zariski open subset of $P(m,n)$ such that
any polynomial of $I_\ell$ is irreducible.
Note that $|-M(\xxx,Y)| \in P(m,n)$.
This section stands to show the following fact.

\begin{prop} \label{prop:irr}
The set
$$\{Y\in\RRR^{n\times n\times\ell} \mid |-M(\xxx,Y)| \in I_\ell\}$$
is a nonempty Zariski open subset of\/ $\RRR^{n\times n\times \ell}$.
\end{prop}

Let $f_\ell\colon \RRR^{n\times n\times\ell} \to P(m,n)$ be a map which sends
$(Y_1;\ldots;Y_\ell)$ to $|\sum_{k=1}^\ell x_kY_k+x_mE_n|$.
Note that $|-M(\xxx,Y)| \in I_\ell$ if and only if 
$f_\ell(Y) \in I_\ell$.
Since $I_\ell$ is a Zariski open set,
$$\mathfrak{T}_\ell:=\{Y\in\RRR^{n\times n\times \ell}  \mid 
f_\ell(Y) \in I_{\ell}\}$$
is a Zariski open subset of $\RRR^{n\times n\times \ell}$.
Then it suffices to show that $\mathfrak{T}_\ell$ is not empty.
First, we show it in the case where $m=3$.
\par
%%%%
The affine space $P(3,n)$ is isomorphic to a real vector space of dimension $n(n+3)/2$ with basis
$$\{x_1^ax_2^bx_3^c\mid 0\leq a,b,c\leq n, a+b+c=n, c\ne n\}.$$
Let $G$ be a map from $\RRR^{n\times n\times 2}$ to $\RRR^{n(n+3)/2}$
defined as
$$G((Y_1;Y_2))=\phi(|x_1Y_1+x_2Y_2+x_3E_n|),$$
where $\phi\colon P(3,n)\to \RRR^{n(n+3)/2}$ is an isomorphism.
It suffices to show that the Jacobian matrix of $G$ has generically full column rank.
To show this, we restrict the source of $G$ to
$$S:=\{(Y_1;Y_2) \in \RRR^{n\times n\times 2} \mid
Y_1=
\begin{pmatrix} u_{11} & 0&\cdots & 0\\
u_{21} & u_{22} & \ddots & \vdots \\
\vdots & \ddots & \ddots & 0\\
u_{n1}& \cdots & u_{n-1,1} & u_{n1}
\end{pmatrix},
Y_2=
\begin{pmatrix}0 & 0&\cdots & v_1\\
-1 & 0 & \cdots & v_{2}\\
\vdots & \ddots & \ddots & \vdots \\
0& \cdots & -1 & v_n
\end{pmatrix}
\}$$
of dimension $n(n+3)/2$, say $G|_S\colon S\to \RRR^{n(n+3)/2}$.

\begin{lemma} \label{lem:Jacobian}
The Jacobian of $G|_S$ is nonzero.
\end{lemma}

\begin{proof}
Put $g(Y):=f(Y)-x_3^n$ for $Y\in S$.
Suppose that for constants $c(v_j)$, $c(u_{ij})$, the linear equation
\begin{equation}\label{eqn:1}
\sum_{j=1}^n c(v_j)\frac{\partial g}{\partial v_j}
+\sum_{1\leq j\leq i\leq n} c(u_{ij})\frac{\partial g}{\partial u_{ij}}=0
\end{equation}
holds.
We show that all of $c(v_j)$, $c(u_{ij})$ are zero by induction on $n$.
It is easy to see that the assertion holds in the case where $n=1$.
As the induction assumption, we assume that the
assertion holds in the case where $n-1$ instead of $n$.
We put
$$\lambda_j=u_{jj}x_1+x_3 \text{ and  } \mu(a,b)=\prod_{t=a}^{b}\lambda_t.$$
After a partial derivation, we put $u_{ij}=0$ ($i>j$) and then
have the following equations:

$$
\begin{array}{lcll}
\displaystyle\frac{\partial g}{\partial v_j}&=& 
x_2^{n-j+1}\mu(1,j-1)
& (1\leq j\leq n)\\ 
\displaystyle\frac{\partial g}{\partial u_{jj}}&=& 
x_1\mu(1,j-1)
 \left| \begin{matrix} \lambda_{j+1} && & & v_{j+1}x_2\\
-x_2& \lambda_{j+2} &&& v_{j+2}x_2\\
& \ddots& \ddots&&\vdots \\
& & -x_2 & \lambda_{n-1}& v_{n-1}x_2\\
& && -x_2 & \lambda_{n}+v_nx_2\\
\end{matrix}\right|
& (1\leq j\leq n)\\ 
\displaystyle\frac{\partial g}{\partial u_{ij}}&=& 
\displaystyle -x_1x_2^{n-i}\mu(j+1,i-1)
\left| \begin{matrix} \lambda_{1} & & && v_{1}x_2\\
-x_2& \lambda_{2} &&& v_{2}x_2\\
& \ddots& \ddots&& \vdots \\
& & -x_2 & \lambda_{j-1}& v_{j-1}x_2\\
& && -x_2 & v_jx_2\\
\end{matrix}\right|
& (1\leq j<i\leq n) \\ 
\end{array}
$$

%%%%%%%%%%%%%%%%%%%%%%%%%%%%%%%%%%%%%%%
\noindent
By seeing terms divisible by $\lambda_1$ in the left hand side of
\eqref{eqn:1}, we have
\begin{equation*}\label{eqn:3}
\sum_{j=2}^n c(v_j)\frac{\partial g}{\partial v_j}
+\sum_{2\leq j\leq i\leq n} c(u_{ij})h_{ij}=0,
\end{equation*}
where
$$h_{ij}=\displaystyle -x_1x_2^{n-i}\mu(j+1,i-1))
\left| \begin{matrix} \lambda_{1} & & & & 0\\
-x_2& \lambda_{2} &&& v_{2}x_2\\
& \ddots& \ddots&& \vdots \\
& & -x_2 & \lambda_{j-1}& v_{j-1}x_2\\
& && -x_2 & v_jx_2\\
\end{matrix}\right|.$$
%%%
Note that
$$
\begin{array}{lcll}
\displaystyle\frac{\partial g}{\partial v_j} &=& 
\displaystyle\lambda_1\frac{\partial g^\prime}{\partial v_j} 
& (2\leq j\leq n), \text{ and} \\
h_{ij} &=& 
\displaystyle\lambda_1\frac{\partial g^\prime}{\partial u_{ij}} 
& (2\leq j\leq i\leq n) \\
\end{array}
$$
where $g^\prime$ is the determinant of 
the $(n-1)\times (n-1)$ matrix obtained from 
$x_1Y_1+x_2Y_2+x_3E_n$ by removing the first row and the first column
minus $x_3^{n-1}$.
Therefore by the induction assumption, 
$$c(v_j)=c(u_{ij})=0 \quad (2\leq j\leq i\leq n)$$ 
since
$\displaystyle\frac{\partial g^\prime}{\partial v_j}$, 
$\displaystyle\frac{\partial g^\prime}{\partial u_{ij}}$ 
($2\leq j\leq i\leq n$) 
are linearly independent.
By \eqref{eqn:1}, we have
\begin{equation}\label{eqn:4}
c(v_1)x_2^n
+c(u_{11})\frac{\partial g}{\partial u_{11}}
-\sum_{i=2}^n c(u_{i1})\displaystyle v_1x_1x_2^{n-i+1}\mu(2,i-1)=0.
\end{equation}
By expanding at the $n$-th column, we have
$$
\frac{\partial g}{\partial u_{11}}=\sum_{i=2}^{n-1} v_{i}x_1x_2^{n-i-1}
\mu(2,i-1)+ x_1(\lambda_n+v_nx_2)\mu(2,n-1).
$$
Therefore, the equation \eqref{eqn:4} implies that
\begin{equation*}\label{eqn:5}
c(v_1)x_2^n+
\sum_{i=2}^{n} (c(u_{11})v_i-c(u_{i1})v_1)x_1x_2^{n-i+1}\mu(2,i-1)
+c(u_{11})x_1\mu(2,n)=0.
\end{equation*}
In this equation we notice the coefficients corresponding
to $x_2^s$, $0\leq s\leq n$.
Then we have $c(u_{i1})=c(v_1)=0$ for $1\leq i\leq n$.
%%%%%

Therefore, we conclude that 
$\displaystyle\frac{\partial g}{\partial v_j}$, $\displaystyle\frac{\partial g}{\partial u_{ij}}$ ($1\leq j\leq i\leq n$) are linearly independent, which
means that the Jacobian of $G|_S$ is nonzero.
\qed
\end{proof}

By Lemma~\ref{lem:Jacobian}, there is an open subset
$S$ of $\RRR^{n\times n\times 2}$
such that the rank of the Jacobian matrix of $G$ at $Y$ 
has full column rank for any $Y\in S$.
Then $f_2(S)\cap I_2$ is not empty and
thus $\mathfrak{T}_2 \cap S$ is not empty. 
In particular, $\mathfrak{T}_2$ is not empty.
\par
Now we show that $\mathfrak{T}_\ell$ is not empty in the case where $\ell>2$.
Let $q\colon \RRR^{n\times n\times \ell} \to \RRR^{n\times n\times 2}$
be a canonical projection which sends $(Y_1;\ldots;Y_\ell)$ to
$(Y_{\ell-1};Y_{\ell})$.
Put $\hat{\mathfrak{T}}=q^{-1}(\mathfrak{T}_2 \cap S)$ and
let $\bar{q}\colon P(m,n)\to P(3,n)$ be also a canonical projection which sends a polynomial
$g(x_1,\ldots,x_m)$ to $g(0,\ldots,0,x_{1},x_{2},x_3)$.
The following diagram is commutative.
$$\begin{CD}
\hat{\mathfrak{T}} @>{\subset}>> \RRR^{n\times n\times \ell} @>f_\ell>> P(m,n) \\
@VVV @V{q}VV @V{\bar{q}}VV \\
\mathfrak{T}_2 \cap S @>{\subset}>> \RRR^{n\times n\times 2} @>f_2>> P(3,n)
\end{CD}
$$
Note that if $g(x_1,\ldots,x_m)\in P(m,n)$ is reducible then so is 
$g(0,\ldots,0,x_1,x_2,x_3)\in P(3,n)$.
The set $\hat{\mathfrak{T}}$ is a nonempty open subset of 
$\RRR^{n\times n\times\ell}$ with the property that 
$f_\ell(Y)$ is irreducible for any $Y\in \hat{\mathfrak{T}}$.
Thus $\mathfrak{T}_\ell$ is not empty, 
since $\hat{\mathfrak{T}}\subset\mathfrak{T}_\ell$.
This completes the proof of Proposition~\ref{prop:irr}.
%%%%
%%%%%

\section{Proof of Theorem~\ref{thm:main} \label{sec:proof}}

In this section we show Theorem~\ref{thm:main}.

%%%%%
Let $\check{\xxx}=(x_1,\ldots,x_\ell)^\top$ for $\xxx=(x_1,\ldots,x_\ell,x_m)^\top$, and put
$$\psi(\xxx,Y):=\begin{pmatrix}
(-1)^{n+1}|M(\xxx,Y)_{n,1}| \\
(-1)^{n+2}|M(\xxx,Y)_{n,2}| \\
\vdots \\
(-1)^{n+n}|M(\xxx,Y)_{n,n}| 
\end{pmatrix},\quad
\check{\xxx}\otimes \psi(\xxx,Y):=\begin{pmatrix}
x_1\psi(\xxx,Y) \\
x_2\psi(\xxx,Y) \\
\vdots \\
x_\ell\psi(\xxx,Y) 
\end{pmatrix}$$
and
$$U(Y):=\langle \check{\aaa}\otimes \psi(\aaa,Y) \mid \, |M(\aaa,Y)|=0 \rangle.$$ 
%%%%
\begin{lemma} \label{lem:dimVvsM}
If $\dim U(Y)=p$, then $Y\in \mathfrak{M}$.
\end{lemma}
\begin{proof}
Let $\dim U(Y)=p$. Then there are 
$\aaa_j=(a_{1j},\ldots,a_{mj})^\top\in U(Y)$ for $1\leq j\leq p$
such that
$$B^\prime
=(\check{\aaa}_1\otimes \psi(\aaa_1,Y),\ldots,\check{\aaa}_p\otimes \psi(\aaa_p,Y))$$
is nonsingular.
Note that
$M(\aaa_j,Y)\psi(\aaa_j,Y)=\zerovec$ for $1\leq j\leq p$
and
$$B^\prime=\begin{pmatrix} AD_1\\ \vdots \\ AD_{\ell}\end{pmatrix},$$
where
$A=(\psi(\aaa_1,Y),\ldots,\psi(\aaa_p,Y))$ and
$D_k=\diag(a_{k1},\cdots,a_{kp})$ for $1\leq k\leq \ell$.
Thus $Y\in \mathfrak{M}$.
\qed
\end{proof}

For an $n\times \ell$ matrix 
$C=(\ccc_1,\ldots\ccc_\ell)$, we put
$$g(\xxx,Y,C):=\left| \begin{matrix} M(\xxx,Y)^{<n}\\ \sum_{k=1}^\ell x_k\ccc_k^\top\end{matrix}\right|,$$ 
where $M(\xxx,Y)^{<n}$ is the $(n-1)\times n$ matrix obtained from
$M(\xxx,Y)$ by removing the $n$-th row.

\begin{lemma} \label{lem:dimV}
Let $C=(\ccc_1,\ldots,\ccc_\ell)$ be an $n\times\ell$ matrix.
The following claims are equivalent.
\begin{enumerate}
\item $\dim U(Y)=p$.
\item $g(\aaa,Y,C)=0$ 
for any $\aaa \in \RRR^{m}$ with $|M(\aaa,Y)|=0$ implies $C=O$.
\end{enumerate}
\end{lemma}

\begin{proof}
Let $C=(\ccc_1,\ldots,\ccc_\ell)$ be an $n\times \ell$ matrix.
Put $\ddd=(\ccc_1^\top,\ldots,\ccc_\ell^\top)^\top \in \RRR^p$.
The inner product of this vector $\ddd$ with 
$\check{\aaa}\otimes \psi(\aaa,Y)$ is equal to
$g(\aaa,Y,C)$.
Therefore $\ddd$ belongs to the orthogonal complement of $U(Y)$
if and only if $g(\xxx,Y,C)=0$ 
for any $\aaa \in \RRR^{m}$ with $|M(\aaa,Y)|=0$.
Thus the assertion holds.
\qed
\end{proof}

For any $i$ and $k$ with $1\leq i\leq n-1$ and $1\leq k\leq n$,
let $s^{(k)}_i$ be an elementary symmetric polynomial of degree $i$ with
variables $\alpha_1,\ldots,\alpha_{k-1},\alpha_{k+1},\ldots,\alpha_{n}$.
Put 
$$S_n=\begin{pmatrix} 1&1&\ldots&1\\ 
s^{(1)}_1&s^{(2)}_1&\ldots&s^{(n)}_1\\
s^{(1)}_2&s^{(2)}_2&\ldots&s^{(n)}_2\\
\vdots&\vdots&&\vdots\\
s^{(1)}_{n-1}&s^{(2)}_{n-1}&\ldots&s^{(n)}_{n-1}
\end{pmatrix}.$$

\begin{lemma} \label{lem:vanDerMonde}
The determinant $|S_n|$ of the $n\times n$ matrix $S_n$ is equal to
$$\prod_{1\leq i<j\leq n} (\alpha_i-\alpha_j).$$
In particular, if $\alpha_1,\ldots,\alpha_n$ are distinct each other,
then $S_n$ is nonsingular. 
\end{lemma}

\begin{proof}
For any $i$ and $k$ with $1\leq i\leq n-1$ and $2\leq k\leq n-1$,
let $t^{(k-1)}_i$ be an elementary symmetric polynomial of degree $i$ with
variables $\alpha_2,\ldots,\alpha_{k-1},\alpha_{k+1},\ldots,\alpha_{n}$.
For $1\leq i\leq n-1$ and $1\leq k\leq n$, we have
$s^{(k)}_{i}-s^{(1)}_{i}=(\alpha_{1}-\alpha_{k})t^{(k-1)}_{i-1}$.
Then 
$$|S_n|=\prod_{2\leq k\leq n}(\alpha_{1}-\alpha_{k})
\left|\begin{matrix} 1&1&\ldots&1\\ 
t^{(1)}_1&t^{(2)}_1&\ldots&t^{(n-1)}_1\\
\vdots&\vdots&&\vdots\\
t^{(1)}_{n-2}&t^{(2)}_{n-2}&\ldots&t^{(n-1)}_{n-2}
\end{matrix}\right|.$$
Therefore we have the assertion by induction on $n$.
\qed
\end{proof}

The following lemma is obtained straightforwardly.
\begin{lemma} 
$$\left|\begin{matrix} \alpha_1+z&&&&a_1\\ & \alpha_2+z &&& a_2\\
&&\ddots && \vdots \\
&&&\alpha_{n}+z&a_{n}\\
b_1&b_2&\ldots&b_{n}&0\end{matrix}\right|=-(z^{n-1},z^{n-2},\ldots,1)S_n\begin{pmatrix} a_1b_1\\ a_2b_2\\ \vdots\\ a_nb_n\end{pmatrix}.$$
\end{lemma}

\begin{proof}
We see the left hand of the equation is equal to
\begin{equation*}
\begin{split}
&-\sum_{k=1}^n a_kb_k\frac{\prod_{1\leq i\leq n}(\alpha_i+z)}{\alpha_k+z} \\
&=-\sum_{k=1}^n a_kb_k\left(\sum_{i=1}^{n} s_{i-1}^{(k)}\right) z^{n-i} \\
&=-\sum_{i=1}^{n} \left(\sum_{k=1}^n a_kb_ks_{i-1}^{(k)}\right) z^{n-i} \\
&=-(z^{n-1},z^{n-2},\ldots,1)\begin{pmatrix} \sum_{k=1}^n a_kb_k\\
\sum_{k=1}^n a_kb_ks_{1}^{(k)}\\
\vdots\\ \sum_{k=1}^n a_kb_ks_{n-1}^{(k)}\end{pmatrix}. \\
\end{split}
\end{equation*}
\qed
\end{proof}

\begin{cor} \label{cor:3slice}
Let $\alpha_1,\ldots,\alpha_{n-1}$ be distinct complex numbers,
$a_1,\ldots,a_{n-1}$ nonzero complex numbers, and
$b_1,\ldots,b_{n-1}$ complex numbers.
If 
$$\left|\begin{matrix} \diag(\alpha_1,\ldots,\alpha_{n-1})+zE_{n-1} & \aaa\\
\bbb^\top & 0\end{matrix}\right|=0$$ 
for any $z \in \RRR$,
then
$\bbb=\zerovec$, where
$\aaa=(a_1\ldots,a_{n-1})^\top$ and $\bbb=(b_1,\ldots,b_{n-1})^\top$.
\end{cor}

\begin{proof}
Since $S_n\begin{pmatrix} a_1b_1\\ a_2b_2\\ \vdots\\ a_nb_n\end{pmatrix}=\zerovec$ and $S_n$ is nonsingular, 
we have $(a_1b_1,\ldots,a_nb_n)=\zerovec^\top$.
\qed
\end{proof}

The set
$$\mathfrak{U}_1=\{Y\in\RRR^{n\times n\times \ell} \mid
\, |M(\xxx,Y)| \text{ is irreducible} \}$$ 
is a nonempty Zariski open subset of 
$\mathbb{R}^{n\times n\times \ell}$ (see Proposition~\ref{prop:irr}).
%%%%%
Let $W$ be the subset of $\RRR^{n\times n}$ consisting of
matrices $\begin{pmatrix} A_1&A_2\\ A_3&A_4\end{pmatrix}$ such that
all eigenvalues of $A_1$ are distinct over the complex number field and
every element of the vector $P^{-1}A_2$ is nonzero complex number where
$A_1\in \RRR^{(n-1)\times (n-1)}$, $P\in \CCC^{(n-1)\times (n-1)}$
with $P^{-1}A_1P$ is a diagonal matrix. 
Note that the validity of the condition that every element of the vector $P^{-1}A_2$ is nonzero is independent of the choice of $P$.
We put
$$\mathfrak{U}_2:=\{(Y_1;\ldots;Y_\ell)\in \RRR^{n\times n\times \ell}\mid 
Y_k \in W, 1\leq k\leq\ell\}.$$ 
The set $\mathfrak{U}_2$ is a nonempty Zariski open subset of 
$\mathbb{R}^{n\times n\times \ell}$ 
and 
$\mathfrak{U}:=\mathfrak{U}_1\cap \mathfrak{U}_2$ is also.
%%%%%

\begin{lemma} \label{lem:detzero.cond}
Let $Y\in \mathfrak{U}_2$ and $\ddd_1,\ldots,\ddd_\ell\in\RRR^{n-1}$.
If 
$$\left|\begin{matrix} \multicolumn{2}{c}{M(\aaa,Y)^{<n}} \\ \sum_{k=1}^\ell a_k \ddd_k^\top & 0\end{matrix}\right| = 0$$ 
for any $\aaa=(a_1,\ldots,a_m)^\top \in \RRR^m$,
then $\ddd_1=\cdots=\ddd_\ell=\zerovec$.
\end{lemma}

\begin{proof}
Let $1\leq k\leq \ell$.
Take $a_k=1$ and $a_j=0$ for $1\leq j\leq\ell$, $j\ne k$ and
put $Y_k=\begin{pmatrix} A_1&A_2\\ A_3&A_4\end{pmatrix}$, where
$A_1$ is an $(n-1)\times(n-1)$ matrix.
Since $Y_k\in W$, there are a matrix $P\in \CCC^{(n-1)\times (n-1)}$ 
and distinct complex numbers $\alpha_1,\ldots, \alpha_{n-1}$
such that
$$\diag(P,1)^{-1}\begin{pmatrix} \multicolumn{2}{c}{(Y_k-a_mE_n)^{<n}} \\ 
\ddd_k^\top & 0\end{pmatrix}\diag(P,1)=
\begin{pmatrix} \diag(\alpha_1,\ldots,\alpha_{n-1})-a_mE_{n-1} & P^{-1}A_2\\
\ddd_k^\top P & 0 \end{pmatrix}$$
and every element of $P^{-1}A_2$ is nonzero.
Then we have $\ddd_k^\top P=\zerovec^\top$ by Corollary~\ref{cor:3slice} and
thus $\ddd_k=\zerovec$.
\qed
\end{proof}

The following lemma is essential for the proof of Theorem~\ref{thm:main}.

\begin{lemma} \label{lem:UvsC2}
$\mathfrak{U}\cap \mathfrak{C}\subset \mathfrak{M}$.
In particular, $\overline{\mathfrak{C}}\subset\overline{\mathfrak{M}}$ holds.
\end{lemma}

\begin{proof}
Let $Y\in \mathfrak{U}\cap \mathfrak{C}$ and fix it.
There exists $\aaa=(a_1,\ldots,a_\ell,a_m)^\top$
such that $|M(\aaa,Y)|<0$. 
Then there is an open neighborhood $U$ of
$(a_1,\ldots,a_\ell)^\top$ and a mapping $\mu\colon U\to \RRR$ such that
$$|M(\begin{pmatrix} \yyy\\ \mu(\yyy)\end{pmatrix},Y)|=0$$
for any $\yyy \in U$.
Thus $|M(\xxx,Y)|=0$ determines an $(m-1)$-dimensional algebraic set.
Let $C$ be an $n\times \ell$ matrix.
Now suppose that $g(\aaa,Y,C)=0$ holds for any $\aaa\in \RRR^m$ with $|M(\aaa,Y)|=0$.
We show that $g(\xxx,Y,C)$ is zero as a polynomial 
over elements of $\xxx$.
As a contrary, assume that $g(\xxx,Y,C)$ is not zero.
The degree of $g(\xxx,Y,C)$ corresponding to the $m$-th
element of $\xxx$
is less than $m$ which is that of $|M(\xxx,Y)|$.
Furthermore, since $M(\xxx,Y)$ is irreducible, $M(\xxx,Y)$ and $g(\xxx,Y,C)$ are coprime.
Then there are
polynomials $f_1(\xxx)$, $f_2(\xxx)\in\RRR[x_1,\ldots, x_\ell, x_m]$ and a nonzero polynomial $h(\check{\xxx})\in\RRR[x_1,\ldots, x_\ell]$ 
such that
$$f_1(\xxx)M(\xxx,Y)
  +f_2(\xxx)g(\xxx,Y,C)=h(\check{\xxx})$$
as a polynomial over elements of $\xxx$, by Euclidean algorithm.
However, we can take $\bbb\in U$ so that $h(\bbb)\ne 0$.
Then the above equation does not hold at $\xxx=\begin{pmatrix} \bbb\\ \mu(\bbb)\end{pmatrix}$.
Hence $g(\xxx,Y,C)$ must be the zero polynomial over elements of $\xxx$.
%%%%%%%
Let $\ccc_k^\top=(c_{1k},\ldots,c_{nk})$.
By seeing the coefficient of $x_m^{n-1}x_k$, we get $c_{nk}=0$ for $1\leq k\leq\ell$.
Therefore $C=O$ by Lemma~\ref{lem:detzero.cond}.
By Lemmas~\ref{lem:dimV} and \ref{lem:dimVvsM} we get $Y\in\mathfrak{M}$.
Therefore $\mathfrak{U}\cap \mathfrak{C}$ is a subset of $\mathfrak{M}$.
%%%%%
Then $\overline{\mathfrak{C}}=\overline{\mathfrak{U}\cap \mathfrak{C}}
\subset \overline{\mathfrak{M}}$. 
\qed
\end{proof}

\begin{thm} \label{thm:SvsUvsC2}
$\overline{\mathfrak{S}}=\overline{\mathfrak{M}}=\overline{\mathfrak{C}}$ holds.
\end{thm}

\begin{proof}
We have $\overline{\mathfrak{M}}\subset \overline{\mathfrak{S}}$
by Proposition~\ref{prop:UsubsetS}.
By Propositions~\ref{prop:C1capU} and \ref{prop:C3isboundaryC2},
the set $\mathfrak{S}$ is a subset of $\overline{\mathfrak{C}}$
and then $\overline{\mathfrak{S}}\subset\overline{\mathfrak{C}}$.
Therefore $\overline{\mathfrak{S}}=\overline{\mathfrak{M}}=\overline{\mathfrak{C}}$ by Lemma~\ref{lem:UvsC2}.
\qed
\end{proof}

%%%%%

%%%%%%%%%%%%%%%%%%%%%
%%%%%%%%%%%%%%%%%%%%%

\medskip
\noindent{\bf Proof of Theorem~\ref{thm:main}.}
For almost all $Y\in\mathfrak{A}$, $\rank\, X(Y)=p+1$ by Theorem~\ref{thm:ans}.
Since $\mathfrak{A}$ is an open set, if $\mathfrak{A}$ is not an empty set,
then $\typicalrankR(m,n,p)=\{p,p+1\}$ (\cite[Theorem~3.4]{Sumi-etal:2010a}).
Suppose that $\mathfrak{A}$ is empty.
Then $\overline{\mathfrak{M}}=\mathbb{R}^{n\times n\times \ell}$
and the closure of the set consisting of all $n\times p\times m$
tensors equivalent to $X(Y)$ for some $Y\in \mathfrak{M}$
is $\mathbb{R}^{n\times p\times m}$.
Recall that any tensor $X(Y)$ for $Y\in\mathfrak{M}$ has rank $p$.
By Theorem~\ref{thm:Friedland},  $p$ is the maximal
typical rank of $\mathbb{R}^{n\times p\times m}$.
Therefore, 
$$\typicalrankR(m,n,p)=\typicalrankR(n,p,m)=\{p\}$$ 
holds.
\qed

%%%%%%%%%%%%%%%%%%%%%%%%%%%%%%%%
%%%%%%%%%%%%%%%%%%%%%%%%%%%%%%%%

\end{document}